\documentclass[12pt,twoside]{article} 

\usepackage{amsmath,amsthm}

\usepackage{amssymb} 
\usepackage{enumerate} 
\usepackage{hyperref} 
\usepackage[all]{xy}

\newcommand{\id}{\mathfrak}  

\newcommand{\A}{\mathbb{A}}

\newcommand{\yy}{\mathit{y}}

\newcommand{\hilb}{{\mathcal{H}\textnormal{ilb}}}

\newcommand{\Gr}{Gr\"obner }
 
\newcommand{\T}{\mathbb{T}} 
 
\newcommand{\PP}{\mathbb{P}} 
\newcommand{\ZZ}{\mathbb{Z}}

\date{} 

\begin{document} 




\centerline {\Large{\bf HOMOGENEOUS VARIETIES}} 

\centerline{} 

\centerline{\Large{\bf FOR HILBERT SCHEMES}} 

\centerline{} 

\centerline{\bf {Giorgio Ferrarese}} 

\centerline{} 

\centerline{Dipartimento di Matematica dell'Universit\`{a}} 

\centerline{Via Carlo Alberto 10} 

\centerline{10123 Torino, Italy} 

\centerline{ {\small giorgio.ferrarese@unito.it}} 

\centerline{} 

\centerline{\bf {Margherita Roggero}} 

\centerline{} 

\centerline{Dipartimento di Matematica dell'Universit\`{a}} 

\centerline{Via Carlo Alberto 10} 

\centerline{10123 Torino, Italy} 

\centerline{ {\small margherita.roggero@unito.it}} 

\newtheorem{Theorem}{\quad Theorem}[section] 

\newtheorem{Definition}[Theorem]{\quad Definition} 

\newtheorem{Corollary}[Theorem]{\quad Corollary} 
\newtheorem{prop}[Theorem]{\quad Proposition} 

\newtheorem{Lemma}[Theorem]{\quad Lemma} 

\newtheorem{Example}[Theorem]{\quad Example} 
\newtheorem{remark}[Theorem]{\quad Remark} 

\begin{abstract} The paper concerns  the affine varieties that are homogeneous with respect to a (non-standard)  graduation over the group $\ZZ^m$. Among the other properties it is  shown that every such a variety  can be embedded in its Zariski tangent space at the origin, so that it is   smooth if and only if it is   isomorphic to an affine space. The results directly apply to the study of  Hilbert schemes of subvarieties in $\PP^n$.
\end{abstract} 



\section{Introduction} 
In the present paper we collect and analyze some properties of the affine varieties that are homogeneous with respect to a   graduation over the group $\ZZ^m$. Our attention on this kind of varieties arises from the strict   connection,  recently highlighted,    with   the Hilbert schemes  $\hilb^n_{p(z)}$ of subvarieties in $\PP^n$ (that is saturated, homogeneous ideals in $k[x_0, \dots, x_n]$) with Hilbert polynomial $p(z)$. 

Since any ideal in $k[x_0, \dots, x_n]$ has the same Hilbert function as the monomial ideal which is  its initial ideal with respect to any fixed term ordering (see \cite{Green}), we can consider in $\hilb^n_{p(z)}$  the equivalence classes containing all the ideals having the same initial ideal. In \cite{NS} it is proved that the equivalence classes (called \Gr strata) are in fact locally closed subvarieties in $\hilb^n_{p(z)}$ and that each of them can be algorithmically realized, using Buchberger characterization of \Gr bases, as an affine variety in a suitable affine space $\A^N_k$. However both $N$ and the number of equations defining the \Gr stratum of a given monomial ideal  are in general very big, so that an explicit computation can be really heavy.

A considerable improvement of the computational weight as well as -some interesting theoretical consequences can be obtained thanks to the natural homogeneous structure of each \Gr stratum introduced in \cite{RT} (for the zero-dimensional case, see also \cite[Corollary 3.7]{robbiano-2008}).  Especially all the results obtained in the present paper can be usefully applied in order to obtain both geometrical properties of Hilbert schemes and explicit equations defining $\hilb^n_{p(z)}$ for a fixed polynomial $p(z)$.

\medskip 

In \S 1 we recall the general definition of homogeneous ideals and affine varieties with respect to a graduation over the group $\ZZ^m$ and show that this structure respects the usual operations on ideals and the primary decomposition: this allows us to apply our results also to each irreducible component and to each component of its support, when a \Gr stratum lives on more than one connected component   or on a non-reduced component of $\hilb^n_{p(z)}$.

\medskip

In \S 3 we show that such an homogeneous affine  variety $V$ can be  embedded in an affine space that can be, in a natural way, identified with the Zariski tangent space to $V$ at the origin; this is of course the minimal affine space in which $V$ can be isomorphically embedded. In the case of a \Gr stratum using this minimal embedding we can obtain the maximal reduction on the number of involved parameters.

\medskip

 In \S 4 we consider the action over a homogeneous  variety $V$ of a suitable torus induced by the graduation. The orbits of this action give a covering of $V$ by locally closed rational subvarieties, so that $V$ is rationally chain connected; the union of all the fibers of the same dimension give a locally closed stratification of $V$.  Moreover homogeneous cycles generate the  Chow groups  so that in low dimension $A_i(V)$ is generated by the $i$-dimensional fibers (more precisely, by the classes of the closure of those fibers); for instance $A_1(V)$ is always generated by the 1-dimensional fibers. One of the most important results about Hilbert schemes, the connectedness, is   obtained using chains of curves  connecting points on $\hilb_{p(z)}^n$ that correspond to monomial ideals (see \cite{Hcon} for the original proof or \cite{PS} for a different one); the rational curves on $\hilb_{p(z)}^n$ that are the closure of 1-dimensional orbits on \Gr strata could form a \lq\lq connecting net\rq\rq\ over $\hilb_{p(z)}^n$ allowing algorithmic procedures with a reduced computational weight.

\section{$\lambda$-homogeneous affine varieties}

In this section we consider a polynomial ring $k[\yy]:=k[y_1,\dots,y_s]$ over an algebraically closed field $k$. We will denote by   $\T_\yy$ the set of monomials in $k[\yy]$ and by $\overline{\T}_\yy$  the set of monomials   in  the field $k(\yy)$, that is the set of monomials with integer exponents $\yy^{\alpha}=y_1^{\alpha_1}\cdots  y_s^{\alpha_s}$, $\alpha_i \in \ZZ$: of course    $\overline{\T}_\yy $, with the usual product, is the free abelian  group on the set of   variables $y_1, \dots, y_s$. 

For general facts about gradings we refer to \cite{KrRob2} Ch. 4.

\begin{Definition} Let  $G$ be the abelian group $\ZZ^m$ and  $(g_1,\ldots,g_s)$  be a $s$-tuple  of elements in $G$, not necessarily distinct.
The group homomorphism:
\begin{equation}
 \lambda:\  \overline{\T}_\yy \longrightarrow G \hbox{\ \ \ given by \ \ \ } 
          \  y_i\, \longmapsto g_i
\end{equation}
induces a  graduation
$$\displaystyle k[\yy]=\bigoplus_{g\in G}k[\yy]_g$$
where 
for every $g\in G$ the $\lambda$-homogeneous component $k[\yy]_g$  is the $k$-vector space generated by all the monomials $\yy^{\alpha} \in \T_\yy$ such that $\lambda(\yy^{\alpha})=g$.

A polynomial $F\in k[\yy]$ is $\lambda$-homogeneous of $\lambda$-degree $g$ if $F\in k[\yy]_g$.
A $\lambda$-\emph{homogeneous} ideal  $\id{a}$  is a proper ideal ($\id{a}\neq k[\yy]$)   generated by $\lambda$-homogeneous polynomials. A $\lambda$-\emph{cone} is the subvariety $V=\mathcal{V}(\id{a})$ defined by a $\lambda$-homogeneous ideal.
\end{Definition}

\begin{Lemma}\label{ordine} In the above notation the following are equivalent:
\begin{enumerate}
\item $k[y]_{0_G}=k$;
\item  the group $G=\ZZ^m$ can be equipped with a structure   of totally ordered group   in such a way that $\lambda(y_i)\succ 0_G$, for $i=1, \dots, s$.
\end{enumerate}
\end{Lemma}
\begin{proof} \textit{2.} $ \Rightarrow  $ \textit{1.} For every non constant monomial $y^{\alpha}$  (that is $\alpha=[\alpha_1, \dots, \alpha_s]$, $\alpha_i \geq 0$ and   $\alpha_{i_0}> 0$) we have $\lambda(y^{\alpha})\geq \alpha_{i_0}\lambda(y_{i_0})\succ 0_G$ and then $y^{\alpha}\notin k[y]_{0_G}$.  

\medskip

\textit{1.} $ \Rightarrow  $ \textit{2.} Let $W=\{c_1 \lambda(y_1)+ \dots +c_s \lambda(y_s) \  / \ c_i \in \mathbb{R}, \ c_i\geq 0 \}$. The condition\textit{ 1.} insures that $W$ does not contain any couple of non-zero opposite vectors $v=c_1 \lambda(y_1)+ \dots +c_s \lambda(y_s)$ and $-v=c'_1 \lambda(y_1)+ \dots +c'_s \lambda(y_s)$; in fact if there is  such a couple of vectors  with real coefficients $c_i,c'_i$, then there is also a couple of them with integer, non negative, coefficients, so that $y_1^{c_1+c'_1}\cdots y_s^{c_s+c'_s}$ would be a non-constant monomial in $k[y]_{0_G}$.

Then there exists an hyperplane $\pi$ in $\mathbb{R}^m$  meeting $W $ only in the origin. Using an orthogonal vector $\omega$ to $\pi$ we can define a total order $\preceq$ in $\ZZ^m$ such that $\lambda(y_i)\succ 0_G$ for every $i=1, \dots,s$: $y^{\alpha}\preceq y^{\beta}$ if $\omega \cdot \lambda(y^{\alpha}) \geq \omega \cdot \lambda(y^{\beta})$ and using any term ordering as a tie breaker (see \cite{Sturm} Ch 1). 
\end{proof}

It is clear by the previous lemma that the assumption $k[\yy]_{0_G}=k$ insures that  every $\lambda$-homogeneous ideal is contained in the only $\lambda$-homogeneous maximal ideal $\id{M}=(y_1, \dots, y_s)$ and so the origin $O$ belongs to every $\lambda$-cone.
\begin{Example}
The usual graduation on $k[\yy]$ is given taking $G=\ZZ$ and  $g_i=1$. More generally, if $G=\ZZ$ and $n_i$ are positive integers, the homomorphism $\lambda(y_i)=n_i $ gives the weighted graduation.
\end{Example}

In this paper we  always assume that $\lambda$ satisfies   the equivalent conditions given in  Lemma \ref{ordine}and use the complete terminology  \lq\lq $\lambda$-graduation\rq\rq, \lq\lq $\lambda$-degree\rq\rq, \dots,   leaving the general terms graduation,  degree, \dots, for the standard graduation on $k[y]$, where all the variables have degree 1.


\begin{Lemma} \label{primari}\begin{enumerate}
\item If $\id{a}$, $\id{b}$ are $\lambda$-homogeneous ideals, then  $\id{a}+\id{b}$,  $\id{a}\cap \id{b}$, $\id{a}  \id{b}$  and $\sqrt{\id a}$ are $\lambda$-homogeneous.

\item if $\id{a}$ is the ideal generated by all the   $\lambda$-homogeneous elements in a primary ideal  $\id{q}$, then $\id{a}$  is primary.

\item If  $\id{a}$ is $\lambda$-homogeneous, then it has a primary decomposition  given by $\lambda$-homogeneous primary ideals.
\end{enumerate}
\end{Lemma}

\begin{proof} \textit{1.} is quite obvious. We only verify the statement about $\sqrt{\id a}$. Let fix any total order in $G$ as in Lemma  \ref{ordine}; 
it is sufficient to prove that for every  $F \in \sqrt{\id{a}}$, its $\lambda$-homogeneous component of maximal $\lambda$-degree  $F_m$ belongs to $\sqrt{\id{a}}$. By definition there is a suitable integer $r$ such that  $F^r\in \id{a}$. The maximal $\lambda$-homogeneous component of  $F^r$ is  $(F_m)^r$ which belongs to  ${\id{a}}$ because ${\id{a}}$  is  $\lambda$-homogeneous; then    $F_m \in \sqrt{\id{a}}$.

\medskip

\textit{2.}  Let  $F,\ G$ be polynomials  such that  $FG\in \id{a}$, but $F\notin \id a$. We have to prove that $G^r$ belongs to $\id a$ for some integer $r$. As $\id a$ is $\lambda$-homogeneous, it is sufficient to prove this property assuming that    $F$ and $G$ are  $\lambda$-homogeneous. So we have $FG\in \id q$ because  $\id{a} \subseteq \id q$ and $F\notin \id q$ (in fact if $F\in \id q$ then also  $F\in \id a$, because both $F$ and $\id a$ are $\lambda$-homogeneous). By the hypothesis, $\id q$ is a primary ideal;  then  $G^r \in \id q$ for some integer $r$ so that  $G^r \in \id a$, because $G^r$ is $\lambda$-homogeneous. 

\medskip

\textit{3.} Let  $ \bigcap \id{q}_i$ be a primary decomposition of   $\id{a}$ and denote by  $\overline{\id{q}_i}$ the ideal generated by all the homogeneous elements in $\id{q}_i$: thanks to the previous item, we know that  $\overline{\id{q}_i}$ is  primary. Moreover $\id{a}=\bigcap \id{q}_i  \supseteq \bigcap\overline{  \id{q}_i}   $.  On the other hand if $F\in \id{a}$ is    $\lambda$-homogeneous then   $F\in {\id{q}_i}$  and so   $F\in \overline{\id{q}_i}$ for every $i$. Thus  we obtain the opposite inclusion  $\id{a} \subseteq \bigcap\overline{  \id{q}_i}$, because  $\id{a}$ is generated by its $\lambda$-homogeneous elements.

\end{proof}

\begin{remark} If $\id {q}$ is an isolated component  of a $\lambda$-homogeneous ideal $\id a$,  then $ \id{q}$ is $\lambda$-homogeneous, because  isolated components are uniquely determined. However, not every   embedded component is necessarily $\lambda$-homogeneous.
\end{remark}

\begin{Example} Let $\id{a}$ be the homogeneous ideal $(x^2,xy)$ in $k [x,y]$ with respect to the usual graduation. Then $\id{a}=(x)\cap (x^2,y)=(x)\cap(x^2,xy,y^r+x)$ ($r\geq 2$) has  a homogeneous primary decompositions and also a primary decompositions having a non-homogeneous  embedded component.
\end{Example}

\begin{Corollary}
Let $V$ be a  $\lambda$-cone in $\A^s$ defined by the $\lambda$-homogeneous ideal $\id a$ in $k[\yy]$. Then
 $V_{\emph{red}}$ 
 and every irreducible component of $V$ are  $\lambda$-cones.
\end{Corollary}

\section{Minimal embedding of a $\lambda$-cone}

The classical definition of the Hilbert scheme $\hilb^{n}_{p(z)}$ realizes it as a closed subvariety of a Grassmannian with a \lq\lq very big\rq\rq\ dimension; also the above   quoted stratification  by \Gr strata gives a locally closed covering of $\hilb^n_{p(z)} $ by homogeneous affine varieties in suitable \lq\lq big\rq\rq\ affine spaces. In the present section we will see that for every homogeneous affine variety there is an  embedding  in an affine space of minimal dimension, that we can identify with the Zariski tangent space at the origin. 

  More precisely if  $V$ is a $\lambda$-cone in $\A^s$,  we define some special linear subspaces $\A^d$ in $\A^s$,    such that the projection $\pi \colon \A^s\to\A^d$ has  interesting properties with respect to $V$.   First of all, $\pi$ is $\lambda$-homogeneous, so that, with respect to the graduation $\A^d$  induced by $\lambda$,  $\pi(V)$ is a $\lambda$-cone and  $\pi \colon V \to \pi(V)$ is an isomorphism. Furthermore if we choose such a linear subspace $\A^d$ of minimal dimension, then  it can be identified  with the Zariski tangent space to $\pi(V)$ at the origin.  All the objects that we have just  described (i.e. $\A^d$, $\pi$, $\pi(V)$, $\dots$) are very easy to obtain  from both a theoretical and a computational point of view.

\begin{Definition} Let   $\id{a}$ be a $\lambda$-homogeneous ideal in $k[y]=k[y_1, \dots,y_s]$. We will denote by $L(\id{a})$ the  $k$-vector space of the linear forms that are the degree 1 homogeneous component of some element in $\id a$ (here \lq\lq homogeneous\rq\rq is related to the usual graduation of $k[y_1,\ldots,y_s]$, that is the $\ZZ$-graduation with all variables of degree 1) and by $T(V)$ the linear subvariety in $\A^s$ (and also $\lambda$-cone) defined by the ideal generated by $L(\id{a})$.
\end{Definition}

A set of generators for $L(\id{a})$ can be simply obtained as the degree 1 homogeneous component of the polynomials in any set of generators of $\id a$. 

\begin{Theorem}\label{elim_th}  
  Let $V$ be a  $\lambda$-cone in $\A^s$ defined by a $\lambda$-homogeneous ideal $\id a$ in $k[\yy]$. Consider any subset $\{y''\} $  of  $d$  variables  in the set of variables $\{\yy\} $ such that   $L(\id{a})\cup \{ y''\} $ generates the $k$-vector space of the linear forms in $k[y]$ and let $\A^s$ the affine space with coordinates $\yy''$ and the graduation induced by $\lambda$.
  
   The projection $\pi\colon \A^s \to \A^d$  induces a $\lambda$-homogeneous isomorphism  $V\simeq \pi(V)$.
  
 If moreover $d=s-\emph{dim}_k(L(\id{a}))$, then $\A^s$ is the Zariski tangent space $T_O(\pi(V))$ of  $\pi(V)$ at the origin. 
 \end{Theorem}

\begin{proof} By hypothesis there are $\lambda$-homogeneous linear  forms $B_1, \dots, B_e$  in $L(\id{a})$ (where  $e=s-d$) such that $\{ B_1, \dots, B_e\}\cup y''$ is  a base for the $k$-vector space of the linear forms in $k[y]$. Then  $\id{a}$ has a set of $\lambda$-homogeneous generators   of the  type:
\begin{equation}\label{gen_a}
B_1 + Q_1 \ ,\ \ldots\ , \ B_e+ Q_e\ ,\ F_{1}\ ,\ \ldots\ ,\ F_{n} 
\end{equation} 
where   $Q_i, F_j\in k[\yy'']$ so that the  inclusion (which is the algebraic translation  of $\pi \colon V \to\pi(V)$): 
\begin{equation}\label{eliminate}
k[\yy'']/(F_{1}\, \ldots, F_{n}) \hookrightarrow
 k[ \yy]/\id{a}.
\end{equation}
  is in fact an isomorphism (see \cite{RT}, Proposition 2.4).
  
  If $d=s-\emph{dim}_k(L(\id{a}))$ then  $B_1, \dots, B_e$  generate $L(\id{a})$ so that  $F_j\in (y'')^2k[\yy'']$ and $T_O(\pi(V))$ is a linear space of dimension $d$ in $\A^d$, that  is $\A^d$ itself.  
\end{proof}

Note that the set of  variables $\yy''$ is not necessarily uniquely defined and so the isomorphism obtained in the previous result is  \lq\lq natural\rq\rq, but not canonical. In any case we may summarize the previous result saying  that \textit{every $\lambda$-cone can be embedded in its Zariski tangent space $T_O(V)$}. 

\medskip

As a straightforward consequence of  Theorem \ref{elim_th}  and  especially of (\ref{eliminate}) we obtain the following result.

\begin{Corollary} Let $V$ be a  $\lambda$-cone of dimension $d$ in $\A^s$ defined by a $\lambda$-homogeneous ideal $\id a$. The following are equivalent:
\begin{enumerate}
	\item the origin  is a smooth point for $V$;
	\item $V \simeq T_O(V)$;
	\item $\emph{dim}_k ( L(\id{a}))=d$;
	\item there is a $\lambda$-homogeneous linear subspace $\A^d$ in $\A^s$ such that the projection induces an isomorphism  $V\simeq \A^d$.
\end{enumerate}
\end{Corollary}

\section{The torus action  on a $\lambda$-cone} 
In the present section we will analyze the torus action  on a $\lambda$-cone $V$ defined by a $\lambda$-homogeneous ideal $\id a$ in $k[\yy]=k[y_1, \dots, y_s]$ . We may assume that the dimension $m$ of  the group $G$ is minimal, namely that the subgroup generated by $\lambda(y_1), \dots, \lambda (y_s)$ is a  $m$-dimensional lattice in  $\ZZ^m$.  
 
  If   $\lambda(y_i)=(n_{i1}, \dots, n_{ir})\in \ZZ^m$, we can associate to $\lambda$ the affine toric variety $\mathcal{T}\subseteq \A^r$ parametrically given by:
 \begin{equation} 
 \left\{ \begin{array}{lll} y_1 & = & t^{\lambda(y_1)} \\
 & \dots & \\
 y_s & = & t^{\lambda(y_s)}
  \end{array} \right.
 \end{equation}
 where  $t$ stands for $[t_1, \dots, t_m]$,  $t^{\lambda(y_i)}$ for $t_1^{n_{i1} }\dots t_r^{n_{ir}}$  and, for every $i$, the parameter $t_i $ varies in $ k^*$. Note that by construction the dimension of  $\mathcal{T}$ is precisely  $m$.

 There are natural actions of the torus  $\mathcal{T}$ on both $k[y]$ and $\A^s$ given by:
 \begin{equation} 
 \left\{ \begin{array}{lll} y_1 & \rightarrow  & y_1 t^{\lambda(y_1)} \\
 & \dots & \\
 y_s & \rightarrow &y_s t^{\lambda(y_s)}
  \end{array} \right.
 \end{equation}
 We will denote both of them again by $\lambda_{\mathcal{T}}$.
 
 \begin{remark}  In the above notation:  
 \begin{description}
 \item[i)]  a polynomial  $F\in k[y]$ is  $\lambda$-homogeneous of $\lambda$-degree $g$ $\Longleftrightarrow$ 
 $\lambda_{\mathcal{T}}(F)=  t^gF$;
 \item[ii)] a subvariety $V\subseteq \A^s$ is a $\lambda$-cone  $\Longleftrightarrow$ 
$\lambda_{\mathcal{T}}(V)=V$.
 \end{description}  
 \end{remark}
 
  For every point  $P\in \A^s$ we will denote by $\Lambda_P$ its orbit with respect to  $\lambda_{\mathcal{T}}$. For every $P\in \A^s$,   $\Lambda_P$ is a toric variety of  dimension $m(P)\leq m$; the only point $P$ such that $m(P)=0$ is the origin $P=O$ and, on the other hand, $m(P)=m$ if  no coordinate of $P$ is zero. 

 The following result collects some easy consequences of the definition of orbits and of the properties of $\lambda$-cones already proved.

\begin{prop} In the above notation, let $V$ be a $\lambda$-cone in $\A^s$. Then:
\begin{enumerate}
\item   $\Lambda_P\subset V_{\emph{red}}$ for every $P\in V$;
\item a reduced, closed subvariety $W\subseteq \A^s$ is a    $\lambda$-cone if and only if  it is the union of orbits, namely  $W=\bigcup_{P\in W} \Lambda_P$; 
\item if $P$ belongs to the singular locus $\emph{sing}(V)$ of  $V$, then  $\Lambda_P\subset \emph{sing}(V)$;
\item $\emph{sing}(V)$ is a $\lambda$-cone.
\end{enumerate} 
\end{prop}
  
  \begin{Example} Let us consider the group  $G=\ZZ^s$  and set $\lambda(y_i)=e_i$. If  $P_j$ is any point having  exactly $j$ non-zero coordinates, then its orbit  $\Lambda_{P_j}$ is a $j$-dimensional torus. Its closure   $V=\overline{\Lambda_{P_j}}$  in $\A^s$ is a linear space of dimension $j$, union of orbits,  that give a cellular decomposition  for it.
\end{Example}

\begin{Example}\label{nontorico} Let  us consider $G=\ZZ^2$  with the lexicographic order and put on  $k[y_1,y_2,y_3,y_4]$ the graduation given by $\lambda(y_1)=[1,2]$, $\lambda(y_2)=[1,0]$,  $\lambda(y_3)=[0,1]$, $\lambda(y_4)=[2,3]$: of course all the variables have a positive degree. 

The exceptional orbits are, besides $\Lambda_O$, those of the four points $P_1(1,0,0,0)$, $P_2(0,1,0,0)$, $P_3(0,0,1,0)$, $P_4(0,0,0,1)$: in fact   the dimension of $\Lambda_P$ can drop only if at least 3 of the coordinates vanish because any two different $\lambda(y_i)$,  $\lambda(y_j)$ are linearly independent. Their closure  is $\overline{\Lambda_{P_i}}=\Lambda_{P_i}\cup \Lambda_O$

For a general $P(a_1,a_2,a_3,a_4)\in \A^4$ the orbit $\Lambda_P$ is a 2-dimensional torus $(k_*)^2$. If for instance $a_2,  a_3\neq 0$ its closure   is the $\lambda$-cone given by the ideal $(a_2a_3^2y_1-a_1y_2y_3^2,a_2^2a_3^3y_4-a_4y_2^2y_3^3)$ and $\overline{\Lambda_P}\setminus {\Lambda_P}\subset \bigcup \overline{\Lambda_{Pi}}$. 
\end{Example}

Denote by $\mu(V)$ the maximal dimension of orbits in a $\lambda$-cone $V$.
 As $\mu(V)$ can be strictly lower than $\dim (V)$, orbits do not give in general a cellular decomposition of $V$ (see  Example \ref{nontorico}).
 
\begin{prop} Let $V$ be a $\lambda$-cone and let  $\mu_0$ be  any integer $0\leq \mu_0 \leq \mu=\mu (V)$. Then:
\begin{enumerate}
	\item  the set of points  $P\in V$ such that $m(P)\leq \mu_0$  is a closed  subset    of  $V$ and  a $\lambda$-cone;
	\item if $\mathrm{dim}(V) \geq 1$, then  $V$ contains some 1-dimensional orbit.
\end{enumerate}
\end{prop}

\begin{proof} For \textit{1.} we may assume that $V$ is the affine space $\A^s$. If $P(a_1, \dots,a_s)$ is any point in $\A^s$, then $ m(P) $ is the dimension of the lattice generated by the set $\{\lambda(y_i) \hbox{ s.t. }  a_i\neq 0\}$. Then  $m(P)\leq \mu_0$ if and only if  there is a set of indexes $i_1, \dots, i_h$ such that  $\lambda(y_{i_1}), \dots, \lambda(y_{i_{h}})$ generate a lattice of dimension $\mu_0$ and $a_j=0$ for ever $j\neq i_1, \dots,i_h$. Thus $\{P \hbox{ t.c. } m(P)\leq \mu_0\}$ is the union of suitable intersections of coordinate hyperplanes. 

\medskip

 \textit{2.} is clearly true if  $d=\dim(V)=1$,  because every  orbit in $V$  has dimension $\leq 1$ and the only 0-dimensional orbit is $\Lambda_O$. Then assume $d\geq 2$ and  proceed by induction on  $s$. 

Let $W$ the $\lambda$-cone defined by $\id{a}+(y_s)$ whose dimension $d'$ satisfies the inequality $ d'\geq d-1\geq 1$. We can also consider $W$ as the $\lambda$-cone in $\A^{s-1}$ defined by $(\id{a}+(y_s))\cap k[y_1, \dots.y_{s-1}]$ (with the graduation induced by $\lambda$) and conclude by induction that $W$ contains some 1-dimensional orbit. Then also $V$ does, because  $W\subset V$ and so all the orbits on $W$ are also orbits on $V$. 
\end{proof}

\begin{Example}  Consider the group $G=\ZZ^2$ and the graduation on $k[y_1,y_2,y_3,y_4]$ given by $\lambda(y_1)=[1,0]$ and $\lambda(y_2)=\lambda(y_3)=\lambda(y_4)=[0,1]$. The set of points $P\in \A^4$ with $m(P)=1$ is the union of the hyperplane $y_1=0$ and the 1-dimensional linear space $y_2=y_3=y_4=0$.
\end{Example}

\begin{prop} Let  $V$ be an irreducible  $\lambda$-cone and let  $U$ be the open subset  $V$ of points $P$ such that  $dim(\Lambda_P)=\mu$ is maximal in $V$.  Then we can find $\mu$ variables $y_{i_1}, \dots, y_{i_{\mu}}$ such that  the linear space $H_{{i_1}, \dots,{i_{\mu}}}$ given by   $ y_{i_1}= \dots = y_{i_{\mu}}=1$ meets   $\Lambda_{P}$ in   finitely many points  if  $P\in  U_{{i_1}, \dots,{i_{\mu}}}\cap U$ and does not meet $\Lambda_P$ if $P\in V \setminus  (U_{{i_1}, \dots,{i_{\mu}}}\cap U)$.

If $\lambda(y_{i_1}), \dots, \lambda(y_{i_{\mu}})$ generate the same lattice than $\lambda(y_1), \dots, \lambda(y_s)$, then  $H_{{i_1}, \dots,{i_{\mu}}}$ meets $\Lambda_P$ in exactly one point for every $P\in  U_{{i_1}, \dots,{i_{\mu}}}\cap U$. In this case  on the  open set  $ U_{{i_1}, \dots,{i_{\mu}}} \cap V$ of $V$  the quotient under the torus action  is naturally isomorphic to the affine variety $V_{{i_1}, \dots,{i_{\mu}}}= U\cap H_{{i_1}, \dots,{i_{\mu}}} $. 
\end{prop}
\begin{proof} Up to a permutation of indexes, we can assume that   $V$ is contained in the linear space  $L: y_{h+1}=\dots= y_r=0$ and is not contained in the hyperplane $y_i=0$ for every  $i\leq h$. Then the maximal dimension $\mu$ of orbits in $V$ is the dimension of the  $k$-vector space generated by    $\lambda(y_1), \dots, \lambda(y_h)$. For every choice of  $\mu$ variables such that $\lambda(y_{i_1}) , \dots, \lambda(y_{i_{\mu}})$ are linearly independent;  let $ U_{{i_1}, \dots,{i_{\mu}}}$ be the complement  in $\A^s$ of the union of the hyperplanes $y_{i_j}=0$, $j=1, \dots, \mu$.  A point  $P$ belongs to $ U_{{i_1}, \dots,{i_{\mu}}} $ if and only if it has non-zero  ${i_1} , \dots, i_{\mu}$ coordinates, so that we can find in its orbit a point with every ${i_1} , \dots, i_{\mu}$-coordinate equal to 1; if we take two such points  in $\Lambda_P$,  their  $i$-th coordinates can differ only up to a $r$-th root of 1, where $r$ is the absolute value of the determinant of the matrix with row  $\lambda(y_{i_1}) , \dots, \lambda(y_{i_{\mu}})$.

On the other hand,  $P$ does not belong to  $ U_{{i_1}, \dots,{i_{\mu}}} $ if  some $i_j$-coordinate is $0$ and then the same property holds for every point in its orbit.
\end{proof}

 If we can find sufficiently many \lq\lq good\rq\rq\ sets of $\mu$ variables (such that the corresponding determinant is $1$) and obtain a covering of $U$ by the  open subsets $ U_{{i_1}, \dots,{i_{\mu}}} $, we can  glue together the affine varieties $V_{{i_1}, \dots,{i_{\mu}}}$  and obtain  a scheme that we can consider as the best approximation of a quotient of $V$ over $G$; however  this scheme is not in general a good geometrical object: for instance it is not necessary separated.  For a general discussion of this topic  see \cite{CLS} Ch. 5 \S \ Constructing quotients. 

\medskip

\begin{Corollary} A  $\lambda$-cone $V$ is a (not necessary smooth) rationally chain connected variety. More precisely it is covered by rational curves passing through the origin $O$, smooth outside the origin (though some of them can be contained in $\mathrm{Sing}(V)$).
\end{Corollary}

\begin{proof} It is sufficient to prove the statement assuming that $V$ is the closure of an orbit  $V=\overline{\Lambda_P}$ of a point $P(a_1, \dots, a_s)$. By our general hypothesis on the graduation there is a vector $\omega\in \mathbb{Z}^n$ such that $\lambda(y_i)\cdot \omega =c_i> 0$ for every $i=1, \dots,s$ (see Lemma 2.2). We can assume that $c_1, \dots, c_s$ have no common integral factor (dividing if necessary by common factors); then the rational curve parametrized by $(y_1=a_1t^{c_1}, \dots, y_s=a_st^{c_s})$ is completely contained in $V$.
\end{proof}

Finally let us recall the following result proved in \cite{FMS} under  some additional hypothesis (essentially a finite number of orbits) and in a far more general form in \cite{KK} (Theorem 1.1).

\medskip

\begin{Theorem} If $\id a$ is a $\lambda$-homogeneous ideal in $k[\yy]$, then the Chow group $A_{\bullet}(k[\yy]/\id{a})$ is  generated by $\lambda$-homogeneous  cycles.
\end{Theorem}

\begin{Example} With the same graduation on $k[y_1,\dots,y_4]$ introduced in Example \ref{nontorico},
  let $V$ be the threefold in $\A^4$  defined by the  polynomial $F=y_1^2y_2y_3+y_1y_4+y_2y_3^2y_4$, which is $\lambda$-homogeneous of $\lambda$-degree $[3,5]$.  A general point in $V$ has a 2-dimensional orbit, so that a dense open subset of $V$ is covered by a 1-dimensional family of   $(k_*)^2$. Moreover all the 5 exceptional orbits belong to $V$ and especially $\Lambda_{P_2}\cup\Lambda_{P_3}\cup\Lambda_{O}=s\emph{ing} (V)$.
  
  The Chow group of $V$ is generated by $\lambda$-homogeneous cycles. Then $A_0(V)$ is generated by (the class of) $\overline{\Lambda_O}$ and $A_1(V)$ by $\overline{\Lambda_{P_1}}$ $\dots$,  $\overline{\Lambda_{P_4}}$. Finally $A_2(V)$ is generated for instance by $\overline{\Lambda_{Q_1}}$ and $\overline{\Lambda_{Q_2}}$, where $Q_1=(0,0,1,1)$ and $Q_2=(0,1,0,1)$. In fact   $3\overline{\Lambda_{Q_1}}+5 \overline{\Lambda_{Q_2}}$ is cut out by the $\lambda$-homogeneous hypersurface given by $y_1+y_2y_3^2$ and  the projection of $\A^4$ on the hyperplane $y_4=0$ gives an isomorphism from $V\setminus (\overline{\Lambda_{Q_1}}\cup \overline{\Lambda_{Q_2}})$ to an   open subset of $\A^3$; of course these two generators of $A_2(V)$ are not free, because  $3\overline{\Lambda_{Q_1}}+5 \overline{\Lambda_{Q_2}}$ is equivalent to 0.
\end{Example}

{\bf Received: Month xx, 200x}

\end{document}